\documentclass[leqno,12pt]{amsart} 
\usepackage{amssymb,amsthm,amscd}
\theoremstyle{plain} 
\newtheorem{theorem}{\indent\sc Theorem}[section]
\newtheorem{lemma}[theorem]{\indent\sc Lemma}

\newtheorem{proposition}[theorem]{\indent\sc Proposition}

\theoremstyle{definition} 
\newtheorem{definition}[theorem]{\indent\sc Definition}
\newtheorem{remark}[theorem]{\indent\sc Remark}

\newcommand{\overbar}[1]{\mkern 1.5mu\overline{\mkern-1.5mu#1\mkern-1.5mu}\mkern 1.5mu}
\makeatletter
 
  \@addtoreset{equation}{section}
\makeatother

\begin{document}

\title[Classical weight one forms in Hida families]{Classical weight one forms in Hida families: \\ Hilbert modular case} 

\author[T. Ozawa]{Tomomi Ozawa} 

%
\thanks{ 
$^*$ The author is supported by JSPS, Grant-in-Aid for Scientific Research for JSPS fellows (15J00944). }
\address{
Mathematical Institute, Graduate School of Science, Tohoku University \endgraf
6-3, Aramaki Aza-Aoba, Aoba-ku, Sendai 980-8578 JAPAN
}
\email{sb2m06@math.tohoku.ac.jp}

\maketitle

\begin{abstract}
In this paper, we investigate the number of classical weight one specializations of a non-CM ordinary Hida family of parallel weight Hilbert cusp forms. We give an explicit upper bound on the number of such specializations. 
\end{abstract}

\section{Introduction}
\label{sec:intro}
The purpose of this article is to generalize the first part of Dimitrov and Ghate's paper \cite{DG} to the case of parallel weight Hilbert modular forms. Let $F$ be a totally real field and $p$ an odd prime. We consider a primitive $p$-ordinary Hida family ${\mathcal F}$ of parallel weight Hilbert cusp forms defined over $F$. It is well-known that a specialization of ${\mathcal F}$ at any arithmetic points of weight at least two is a classical (holomorphic) Hilbert cusp form. Contrary to the higher weight specializations, it is still mysterious to know about weight one specializations. In the case of $F={\mathbb{Q}}$, Ghate and Vatsal \cite{GV} proved that such a Hida family admits infinitely many classical weight one specializations if and only if it is of CM type (namely, constructed from a Hecke character of a imaginary quadratic field). Further the number of such forms inside a non-CM family is bound by an explicit constant due to Dimitrov and Ghate (\cite{DG}). The former result in the specialization was generalized to the case of totally real fields by Balasubramanyam, Ghate and Vatsal (\cite{BGV}). Motivated by these works, in this paper, we pursue a generalization of the results of Dimitrov and Ghate. Namely, for a non-CM primitive $p$-ordinary Hida family ${\mathcal F}$, we give an explicit estimate on the number of classical weight one specializations of ${\mathcal F}$. 

We now explain the details of the main results. Let $G_F$ be the absolute Galois group ${\rm Gal}(\overbar{F}/F)$ of a totally real field $F$ and ${\mathfrak n}_0$ an integral ideal of $F$ prime to $p$. Let ${\mathcal F}$ be a primitive $p$-ordinary cuspidal Hida family of tame level ${\mathfrak n}_0$, $\rho_{\mathcal F}: G_F \rightarrow GL_2({\rm Frac}(\Lambda_L))$ the big Galois representation attached to ${\mathcal F}$ where $\Lambda_L$ is a finite integral extension of the Iwasawa algebra $\Lambda \cong {\mathcal O}[[X]]$ (\cite{W88}). We take a $G_F$-stable $\Lambda_L$-lattice ${\mathcal L}$ in ${\rm Frac}(\Lambda_L)^2$ and consider the reduction $\bar{\rho}_{\mathcal F}: G_F \rightarrow GL_2({\mathbb{F}})$ of $\rho_{\mathcal F}$ modulo the maximal ideal of $\Lambda_L$. 
The reduction $\bar{\rho}_{\mathcal F}$ may depend on the choice of a lattice ${\mathcal L}$ but its semi-simplification $\bar{\rho}_F^{\rm ss}: G_F \rightarrow GL_2({\mathbb{F}})$ does not (Section \ref{sec:103}). Due to Ohta \cite{O84} and Rogawski-Tunnell \cite{RT}, any classical weight one Hilbert cuspidal eigenform $f$ gives rise to an irreducible totally odd Artin representation $\rho_f: G_F \rightarrow GL_2({\mathbb{C}})$. The image of $\rho_f$ in $PGL_2({\mathbb{C}})$ is either
\begin{itemize}
\item dihedral, namely a dihedral group $D_{2n}$ of order $2n$ for some integer $n \geq 1$; or
\item exceptional, namely the symmetric group $S_4$, or the alternative groups $A_4$ or $A_5$. 
\end{itemize}
It will be revealed in Section \ref{sec:113} that if ${\mathcal F}$ admits a classical weight one specialization $f$, then the image of $\bar{\rho}_{\mathcal F}^{\rm ss}$ in $PGL_2({\mathbb{F}}$) has to be of the same type as that of $\rho_f$, and hence $\bar{\rho}_{\mathcal F}=\bar{\rho}_{\mathcal F}^{\rm ss}$ is irreducible and unique up to isomorphism. If ${\mathcal F}$ is of CM type then the image of $\bar{\rho}_{\mathcal F}$ in $PGL_2({\mathbb{F}})$ is dihedral. According to the above classification, we will distinguish the arguments for each case. 
\subsection*{Dihedral case}
\label{sec:di}
Suppose that ${\mathcal F}$ satisfies the following properties: 
\begin{itemize}
\item[(P1)] ${\mathcal F}$ is residually of dihedral type: namely, there exists a quadratic extension $K$ of $F$ such that $\bar{\rho}_{\mathcal F} \cong {\rm Ind}_K^F(\bar{\varphi})$ for some character $\bar{\varphi}: G_K={\rm Gal}(\bar{F}/K) \rightarrow {\mathbb{F}}^\times$; 
\item[(P2)] ${\mathcal F}$ has a classical weight one specialization $f$ such that the associated representation $\rho_f: G_F \rightarrow GL_2(O)$ ($O$ is a suitable finite integral extension of ${\mathbb{Z}}_p$) has the following property: $\rho_f(I_{\mathfrak p})$ has order at least three for each prime ${\mathfrak p}$ of $F$ lying over $p$. Here $I_{\mathfrak p} \subset G_F$ is the inertia group at ${\mathfrak p}$. 
\end{itemize}
It follows from (P1) and (P2) that $\rho_f \cong {\rm Ind}_K^F(\varphi)$ is induced by a finite order character $\varphi: G_K \rightarrow O^\times$ which is a lift of $\bar{\varphi}$ (Lemma \ref{rmk:221}). The two conditions further imply that for any prime ideal ${\mathfrak p}$ of $F$ lying over $p$, ${\mathfrak p}$ splits in $K$ and the character $\varphi$ is ramified at exactly one of the two prime ideals ${\mathcal P}$ and ${\mathcal P}^\sigma$ of $K$ lying over ${\mathfrak p}$ (Lemma \ref{lem:202}). We assume that $\varphi$ is ramified at ${\mathcal P}$ and unramified at ${\mathcal P}^\sigma$. We let ${\mathcal Q}=\prod_{{\mathfrak p} \mid p}{\mathcal P}$ and ${\rm Cl}_K({\mathfrak n}_0{\mathcal Q}^\infty)=\varprojlim_{r \geq 1} {\rm Cl}_K({\mathfrak n}_0{\mathcal Q}^r)$ the projective limit of narrow ray class groups ${\rm Cl}_K({\mathfrak n}_0{\mathcal Q}^r)$ of $K$ of modulus ${\mathfrak n}_0{\mathcal Q}^r$. The group ${\rm Cl}_K({\mathfrak n}_0{\mathcal Q}^\infty)$ is of finite order if and only if $(\prod_{{\mathfrak p} \mid p} U_{K, {\mathcal P}})/\overbar{U_K}$ is a finite group (Section \ref{sec:202}). If this is the case, we put 
\begin{align*}
 M' &= |{\rm Cl}_K| \cdot \left|\left(\prod_{{\mathfrak p} \mid p} U_{K, {\mathcal P}} \right)/\overbar{U_K} \right| \cdot \prod_{\substack{{\mathfrak l} \mid {\mathfrak n}_0, \\ \text{split in $K$}}} (q_{\mathfrak l}-1) \cdot \prod_{\substack{{\mathfrak l} \mid {\mathfrak n}_0, \\ \text{inert in $K$}}} (q_{\mathfrak l}+1)  
\end{align*}
and  $M({\mathcal F}, K, f)=p^{{\rm ord}_p(M')}$, where ${\rm ord}_p$ is the $p$-adic valuation normalized such that ${\rm ord}_p(p)=1$. We will prove the following
\begin{theorem}
\label{theo:001}
Let ${\mathcal F}$ be a primitive $p$-ordinary cuspidal Hida family satisfying the conditions {\rm (P1)} and {\rm (P2)}. Then the following two statements hold true:  
\begin{enumerate}
\item If ${\rm Cl}_K({\mathfrak n}_0{\mathcal Q}^\infty)$ is an infinite group, then $K/F$ is a totally imaginary extension and there is a Hida family $\mathcal G$ of CM type by $K$ that has the same tame level and the same residual representation as those of ${\mathcal F}$. 
\item If ${\rm Cl}_K({\mathfrak n}_0{\mathcal Q}^\infty)$ is of finite order, then ${\mathcal F}$ is not of CM type and the number of classical weight one specializations of ${\mathcal F}$ is bounded by $M({\mathcal F}, K, f)$. 
\end{enumerate}
\end{theorem}
If we assume Leopoldt conjecture for $F$ and $p$ and $K/F$ is not totally imaginary, then the group ${\rm Cl}_K({\mathfrak n}_0{\mathcal Q}^\infty)$ is of finite order and the second case (2) always happens. In particular in the case of $F={\mathbb{Q}}$ and $K$ is a real quadratic field, the case (2) always happens. It should be noticed that if $K/F$ is totally imaginary, then ${\rm Cl}_K({\mathfrak n}_0{\mathcal Q}^\infty)$ is always an infinite group (Remark \ref{rmk:291} (2)). 

We show that if ${\rm Cl}_K({\mathfrak n}_0{\mathcal Q}^\infty)$ is infinite, then there is a Hida family $\mathcal G$ having the same tame level and the same residual representation as those of ${\mathcal F}$, that admits infinitely many classical weight one specializations (see Sections \ref{sec:203} and \ref{sec:204} for details). In view of Balasubramanyam, Ghate and Vatsal's result which asserts that a non-CM family admits only finitely many classical weight one specializations (Theorem 3 of \cite{BGV}), we see that $K/F$ is totally imaginary and ${\mathcal G}$ is of CM type (by $K$).  

The first case (1) of Theorem \ref{theo:001} is referred to as ``residually of CM type" in \cite{DG}. In this case if ${\mathcal F}$ is a non-CM primitive $p$-ordinary Hida family satisfying (P1) and the quadratic extension $K/F$ in (P1) is totally imaginary, one can also give a bound on the number of classical weight one specializations in ${\mathcal F}$, which is a generalization of Lemma 6.5 in \cite{DG}. Since ${\mathcal F}$ is not of CM type, there is at least one prime ideal ${\mathfrak l}$ in $F$ that is inert in $K$, prime to ${\mathfrak n}_0p$ and the trace of $\rho_{\mathcal F}({\rm Frob}_{\mathfrak l})$ is non-zero (otherwise ${\mathcal F}$ has CM by $K$). Here ${\rm Frob}_{\mathfrak l}$ is a geometric Frobenius at ${\mathfrak l}$. Let $\lambda_{{\mathcal F}, {\mathfrak l}}$ be the number of height one prime ideals of $\Lambda_L$ which contain ${\rm Tr}\rho_{\mathcal F}({\rm Frob}_{\mathfrak l})$ and sit above prime ideals of $\Lambda$ corresponding to weight one specializations. We put $\lambda_{\mathcal F}={\rm min}\left\{\lambda_{{\mathcal F}, {\mathfrak l}} \mid \text{${\mathfrak l}$ is inert in $K$ and prime to ${\mathfrak n}_0p$} \right\}$. Then we will prove the following
\begin{proposition}
\label{prop:002}
Let ${\mathcal F}$ be a non-CM primitive $p$-ordinary Hida family satisfying {\rm (P1)}. Suppose that the quadratic extension $K/F$ in {\rm (P1)} is totally imaginary. Then the number of classical weight one specializations of ${\mathcal F}$ is bounded by $\lambda_{\mathcal F}$. 
\end{proposition}
Next we consider the exceptional case. 

\subsection*{Exceptional case}
\label{sec:ex}
As the image of $\rho_f$ in $PGL_2({\mathbb{C}})$ for any classical weight one forms $f$ in ${\mathcal F}$ has bounded order, our analysis is simpler than the dihedral case. 
Let $p^r$ be the $p$-part of the class number of $F$ and $t$ the number of the prime ideals of $F$ lying over $p$. 
\begin{theorem} 
\label{theo:003}
Let ${\mathcal F}$ be a non-CM primitive $p$-ordinary cuspidal Hida family such that the image of the residual representation $\bar{\rho}_{\mathcal F}: G_F \rightarrow GL_2({\mathbb{F}})$ in $PGL_2({\mathbb{F}})$ is exceptional. Then ${\mathcal F}$ has at most $a \cdot b$ classical weight one specializations, where
\begin{itemize}
\item $a=1$, except $p=5$ and the type of ${\mathcal F}$ is $A_5$, in which case $a=2;$ and
\item $b=p^r$, except $p=3$ or $5$, in which case $b=2^t \cdot p^r$. 
\end{itemize}
\end{theorem}
Further, under some assumptions on a Hida community $\left\{{\mathcal F} \right\}$, we prove the existence of a classical weight one specialization in some member of $\left\{{\mathcal F} \right\}$. 
\begin{proposition}
\label{prop:004}
Let $p \geq 7$ be a prime number that splits completely in $F$ and $\left\{{\mathcal F} \right\}$ a Hida community of exceptional type which is $p$-distinguished and the residual representation $\bar{\rho}_{\mathcal F}$ is absolutely irreducible when restricted to ${\rm Gal}(\overbar{F}/F(\zeta_p))$. Assume further that the tame level ${\mathfrak n}_0$ of $\left\{{\mathcal F} \right\}$ is the same as the Artin conductor of $\bar{\rho}_{\mathcal F}$. Then $\left\{{\mathcal F} \right\}$ has at least one classical weight one specialization $f$. Moreover, any other classical weight one specialization of $\left\{{\mathcal F} \right\}$ can be written as $f \otimes \eta$, where $\eta: G_F \rightarrow {\mathbb{Q}}_p^\times$ is a $p$-power order character of conductor dividing ${\mathfrak n}_0$. 
\end{proposition}

As mentioned at the beginning of the introduction, Theorems \ref{theo:001} and \ref{theo:003} extend Theorems 6.4 and 5.1 of \cite{DG}, respectively, to the case of Hilbert modular forms. The proof is basically the same as \cite{DG}. When $F={\mathbb{Q}}$ and $K$ is a real quadratic field, the finiteness of ${\rm Cl}_K({\mathfrak n}_0{\mathcal Q}^\infty)$ is a consequence of Leopoldt conjecture. However, for a general totally real field $F$, Leopoldt conjecture is still open and we do not know whether or not the group ${\rm Cl}_K({\mathfrak n}_0{\mathcal Q}^\infty)$ is finite. Therefore, a new analysis around the group ${\rm Cl}_K({\mathfrak n}_0{\mathcal Q}^\infty)$ is necessary to get around assuming Leopoldt conjecture. 

This paper is organized as follows: In Section \ref{sec:100}, we recall basics of ordinary Hida families of Hilbert cusp forms and Galois representations attached to them. In the first two subsections of Section \ref{sec:110}, we discuss specializations (including in weight one) of CM and non-CM families and then state the finiteness result for non-CM families in \cite{BGV}. The rest of the section is a summary of projective images of the residual representations of Hida families and of Artin representations attached to classical weight one forms. In Section \ref{sec:200} we discuss the dihedral case and prove Theorem \ref{theo:001} (Theorem \ref{theo:205}) and Proposition \ref{prop:002} (Proposition \ref{prop:294}). In the last section, we deal with the exceptional case and provide proofs of Theorem \ref{theo:003} (Theorem \ref{theo:301}) and Proposition \ref{prop:004} (Proposition \ref{prop:303}).

\section{Preliminaries on ordinary Hida families}
\label{sec:100}
In this section, we briefly recall basic notions and properties concerning ordinary Hida families of Hilbert modular forms of parallel weight and the associated Galois representations. Details can be found in \cite{H88} and \cite{W88}. Let $F$ be a totally real field of degree $d$ over the field ${\mathbb{Q}}$ of rationals, $O_F$ the ring of integers of $F$, $G_F={\rm Gal}(\overbar{F}/F)$ the absolute Galois group of $F$ and $p$ an odd prime. Throughout the paper, we fix embeddings $\iota_p: \overbar{\mathbb{Q}} \hookrightarrow \overbar{\mathbb{Q}}_p, \ \iota: \overbar{\mathbb{Q}}_p \hookrightarrow {\mathbb{C}}$ and put $\iota_\infty=\iota \circ \iota_p: \overbar{\mathbb{Q}} \hookrightarrow {\mathbb{C}}$. 

\subsection{Weight and character}
\label{sec:101}
We set up the notation basically following \cite{W88}. Let $F_\infty$ be the cyclotomic ${\mathbb{Z}}_p$-extension of $F$ and ${\mathbf{G}}$ the Galois group ${\rm Gal}(F_\infty/F)$. This $\mathbf{G}$ will be the weight space of Hida families. Let ${\mathcal O}$ be the $p$-adic integer ring of a finite extension of ${\mathbb{Q}}_p$ and $\Lambda={\mathcal O}[[\mathbf{G}]]$ the complete group algebra. We fix a topological generator $\gamma$ of ${\mathbf{G}}$. Then $\Lambda$ can be identified with the power series ring ${\mathcal O}[[X]]$ of one-variable by sending $\gamma$ to $1+X$. 
Let $F(\mu_{p^\infty})$ be the Galois extension of $F$ obtained by adjoining all the $p$-power roots of unity to $F$. Then the group ${\rm Gal}(F(\mu_{p^\infty})/F)$ is isomorphic to ${\mathbb{Z}}_p^\times$ and we put
\begin{align*}
 \chi_p: G_F & \rightarrow {\rm Gal}(F(\mu_{p^\infty})/F) \cong {\mathbb{Z}}_p^\times
\end{align*}
the $p$-adic cyclotomic character of $G_F$. We have an isomorphism ${\rm Gal}(F(\mu_{p^\infty})/F) \cong {\mathbf{G}} \times \Delta$ corresponding to the decomposition ${\mathbb{Z}}_p^\times=(1+p{\mathbb{Z}}_p) \times \mu_{p-1}$, where $\mu_{p-1}$ is the set of $(p-1)$-st roots of unity in ${\mathbb{Z}}_p^\times$. Hence we may and do regard $\chi_p$ as a character of ${\mathbf{G}}$. Let
\begin{align*}
 \omega &= \lim_{n \to \infty} \chi_p^{p^n}: G_F \rightarrow {\mathbb{Z}}_p^\times \rightarrow \mu_{p-1}
\end{align*}
be the Teichm\"uller character. For an integer $k$ and a $p$-power order character $\varepsilon: {\mathbf{G}} \rightarrow \overbar{\mathbb{Q}}_p^\times$, let $\varphi_{k, \varepsilon}: \Lambda \rightarrow \overbar{\mathbb{Q}}_p$ be the ring homomorphism induced by $a \mapsto \varepsilon(a)\chi_p(a)^{k-1}$ on ${\mathbf{G}}$. Let $P_{k, \varepsilon}$ denote the kernel of $\varphi_{k, \varepsilon}$. Note that the character $a \mapsto \varepsilon(a)\chi_p(a)^{k-1}$ on ${\mathbf{G}}$ is of finite order if and only if $k=1$. 

Let $L$ be a finite extension of the field of fractions of $\Lambda$ and $\Lambda_L$ the closure of $\Lambda$ in $L$. 
\begin{definition} 
\label{defn:101}
A prime ideal $P$ of $\Lambda_L$ is called an arithmetic point if $P \cap \Lambda=P_{k, \varepsilon}$ for some integer $k \geq 2$ and a finite order character $\varepsilon: {\mathbf{G}} \rightarrow \overbar{\mathbb{Q}}_p^\times$. 
\end{definition}

Let $\chi_\Lambda: G_F \rightarrow \Lambda^\times$ be the $\Lambda$-adic cyclotomic character obtained by composing the canonical surjection $G_F \rightarrow {\mathbf{G}}$ with the map ${\mathbf{G}} \rightarrow \Lambda^\times$ taking $\gamma$ to  $1+X$. 

\subsection{Ordinary Hida families}
\label{sec:102}
Throughout this paper, we fix a non-zero integral ideal ${\mathfrak n}_0$ in $O_F$ prime to $p$ which will be the tame level of Hida families. Let $\psi: G_F \rightarrow \overbar{\mathbb{Q}}_p^\times$ be a totally odd finite order character of conductor dividing ${\mathfrak n}_0p$ that is tamely ramified at all prime ideals ${\mathfrak p}$ of $F$ lying over $p$. 

\begin{definition} 
\label{defn:102}
A $p$-ordinary cuspidal Hida family ${\mathcal F}$ of tame level ${\mathfrak n}_0$ and nebentype $\psi$ is a collection $\left\{c({\mathfrak m}, {\mathcal F}) \right\}_{\mathfrak m}$ of elements in $\Lambda_L$ indexed by the set of non-zero integral ideals ${\mathfrak m}$ in $O_F$, such that for each arithmetic point $P$ of $\Lambda_L$ with $P \cap \Lambda=P_{k, \varepsilon}$, the series
$$\sum_{{\mathfrak m}} \iota(c({\mathfrak m}, {\mathcal F}) \bmod P)N{\mathfrak m}^{-s}$$
is the Dirichlet series associated with a classical $p$-ordinary Hilbert cusp form of parallel weight $k$, level ${\mathfrak n}_0p^{r+1}$ for the order $p^r$ of $\varepsilon$ and nebentype $(\chi_p^{1-k} \psi \chi_\Lambda) \bmod P=\psi \varepsilon \omega^{1-k}$. Here $N{\mathfrak m}$ is the absolute norm of ${\mathfrak m}$. We denote this classical cusp form by $f_P$ or $\varphi_{k, \varepsilon}({\mathcal F})$ and call this the specialization of ${\mathcal F}$ at $P$. 
\end{definition}
The space of $p$-ordinary cuspidal Hida family of tame level ${\mathfrak n}_0$ and nebentype $\psi$ is finitely generated over $\Lambda_L$ (Proposition 1.5.2; p.~552 of \cite{W88}) and is equipped with Hecke operators $T_{\mathfrak l}, U_{\mathfrak p}$ and diamond operators $S({\mathfrak l})$ for prime ideals ${\mathfrak l} \nmid {\mathfrak n}_0p$ and ${\mathfrak p} \mid {\mathfrak n}_0p$ so that the action is compatible with specializations. 

\begin{definition}[cf. p.~552 of \cite{W88}] 
\label{defn:103}
A primitive $p$-ordinary cuspidal Hida family ${\mathcal F}$ of tame level ${\mathfrak n}_0$ and nebentype $\psi$ is an eigenform for the Hecke operators $T_{\mathfrak l}$, $U_{\mathfrak p}$ and $S({\mathfrak l})$ such that for every arithmetic point $P$ of $\Lambda_L$, the specialization $f_P$ is a $p$-ordinary, $p$-stabilized newform of level divisible by ${\mathfrak n}_0$ (see p.~538 of \cite{W88} for the notion of $p$-stabilized newforms). 
\end{definition}
For such a family ${\mathcal F}$, the eigenvalue of $S({\mathfrak l})$ is $\psi({\mathfrak l})$ for each prime ideal ${\mathfrak l}$ prime to ${\mathfrak n}_0p$, and the nebentype $\psi=\psi_{\mathcal F}$ is referred to as ``the central character of ${\mathcal F}$" in \cite{DG}. 

\subsection{Galois representations attached to ordinary Hida families}
\label{sec:103}
To each primitive $p$-ordinary cuspidal Hida family ${\mathcal F}$ of tame level ${\mathfrak n}_0$ and nebentype $\psi_{\mathcal F}$, Wiles associated a continuous irreducible Galois representation $\rho_{\mathcal F}: G_F \rightarrow GL_2(L)$ which is unramified outside ${\mathfrak n}_0p$ and such that
\begin{alignat}{4}
 {\rm Tr}(\rho_{\mathcal F})({\rm Frob}_{\mathfrak l}) &= c({\mathfrak l}, {\mathcal F}), & \ \ {\rm det}(\rho_{\mathcal F})({\rm Frob}_{\mathfrak l}) &= \psi_{\mathcal F}({\mathfrak l})\chi_\Lambda({\rm Frob}_{\mathfrak l}) \label{eq:trace}
\end{alignat}
for all prime ideals ${\mathfrak l} \nmid {\mathfrak n}_0p$ (Theorem 2.2.1; p.~562 of \cite{W88}). Here ${\rm Frob}_{\mathfrak l}$ is a geometric Frobenius at ${\mathfrak l}$. In the case of $F={\mathbb{Q}}$, this is due to Hida \cite{H86}. In particular this tells us ${\rm det}(\rho_{\mathcal F})=\psi_{\mathcal F}\chi_\Lambda$ (apply Chebotarev density theorem). Specializing the determinant at $P_{k, \varepsilon}$, we obtain a character $\psi_{\mathcal F} \varepsilon \omega^{1-k} \chi_p^{k-1}$. If the specialization $\varphi_{k, \varepsilon}({\mathcal F})$ is a classical cusp form, its nebentype is $\psi_{\mathcal F} \varepsilon \omega^{1-k}$ which is in accordance with Definition \ref{defn:102}. Note that if $k=1$, $\varphi_{1, \varepsilon}({\rm det}(\rho_{\mathcal F}))=\psi_{\mathcal F}\varepsilon$ is a finite order character. 

We give an account of residual representation. Let ${\mathfrak m}={\mathfrak m}_{\Lambda_L}$ be the maximal ideal of the local ring $\Lambda_L$ and $\mathbb{F}=\Lambda_L/{\mathfrak m}$ which is a finite field of characteristic $p$. It is known that $\rho_{\mathcal F}$ preserves a $\Lambda_L$-lattice ${\mathcal L}$ in $L^2$. We regard $\rho_{\mathcal F}$ as $\rho_{\mathcal F}: G_F \rightarrow {\rm End}_{\Lambda_L}({\mathcal L})$ and consider the reduction $\bar{\rho}_{\mathcal F}: G_F \rightarrow {\rm End}_{\mathbb{F}}({\mathcal L}/{\mathfrak m}{\mathcal L})$. We show that the reduction $\bar{\rho}_{\mathcal F}$ takes values in $GL_2({\mathbb{F}})$. Let $P$ be a height one prime ideal of $\Lambda_L$. The localization $\Lambda_{L, P}$ of $\Lambda_L$ at $P$ is a valuation ring. Let ${\mathcal L}_P={\mathcal L} \otimes_{\Lambda_L} \Lambda_{L, P}$. Since ${\mathcal L}$ is a $\Lambda_L$-lattice, we have ${\mathcal L}_P \otimes_{\Lambda_{L, P}} L \cong L^2$. Thus ${\mathcal L}_P$ is free of rank two over $\Lambda_{L, P}$ and we obtain $\rho_{\mathcal F}: G_F \rightarrow GL_2(\Lambda_{L, P})$. Reducing $\rho_
{\mathcal F}$ modulo $P$, we have a continuous representation $\rho_{\mathcal F} \bmod P: G_F \rightarrow GL_2(\Lambda_{L, P}/P\Lambda_{L, P})$. Secondly, let $\Lambda_L'$ be the integral closure of $\Lambda/(P \cap \Lambda)$ in $\Lambda_{L, P}/P\Lambda_{L, P}$. Then $\Lambda_L/P$ is contained in $\Lambda_L'$ which is a valuation ring and $\Lambda_{L, P}/P\Lambda_{L, P}$ is the fraction field of $\Lambda_L'$. The inverse image of the maximal ideal of $\Lambda_L'$ under the map $\Lambda_L \rightarrow \Lambda_L/P \subset \Lambda_L'$ is ${\mathfrak m}={\mathfrak m}_{\Lambda_L}$. By inductive hypothesis on the height of $P$, we have a representation $\bar{\rho}_{\mathcal F}: G_F \rightarrow GL_2({\mathbb{F}})$. Since the trace of $\bar{\rho}_{\mathcal F}$ is characterized by (\ref{eq:trace}), its semi-simplification $\bar{\rho}_{\mathcal F}^{\rm ss}: G_F \rightarrow GL_2({\mathbb{F}})$ does not depend on the choice of ${\mathcal L}$. We call $\bar{\rho}_{\mathcal F}^{\rm ss}$ the residual representation of ${\mathcal F}$. 

We also know that $\rho_{\mathcal F}$ is $p$-ordinary, that is, for each prime ${\mathfrak p}$ sitting above $p$, the restriction of $\rho_{\mathcal F}$ to the decomposition group $D_{\mathfrak p}$ at ${\mathfrak p}$ is
\begin{align*}
 \rho_{\mathcal F}|_{D_{\mathfrak p}} & \cong \left(
 \begin{array}{cc}
  {\mathcal E}_{\mathfrak p} & * \\
  0 & {\mathcal D}_{\mathfrak p}
 \end{array}
 \right) 
\end{align*}
where ${\mathcal E}_{\mathfrak p}, {\mathcal D}_{\mathfrak p}: D_{\mathfrak p} \rightarrow \Lambda_L^\times$ are characters with ${\mathcal D}_{\mathfrak p}$ unramified and ${\mathcal D}_{\mathfrak p}({\rm Frob}_{\mathfrak p})=c({\mathfrak p}, {\mathcal F})$ (Theorem 2.2.2; p.~562 of \cite{W88}). This in particular implies that ${\mathcal E}_{\mathfrak p}={\rm det}(\rho_{\mathcal F})$ on the inertia group $I_{\mathfrak p}$ at ${\mathfrak p}$. Let $\overbar{\mathcal E}_{\mathfrak p}: D_{\mathfrak p} \rightarrow {\mathbb{F}}^\times$ denote the reduction of ${\mathcal E}_{\mathfrak p}$ modulo the maximal ideal of $\Lambda_L$. The same for $\overbar{\mathcal D}_{\mathfrak p}$. 
\begin{definition} 
\label{defn:104}
${\mathcal F}$ is said to be $p$-distinguished if $\overbar{\mathcal E}_{\mathfrak p} \neq \overbar{\mathcal D}_{\mathfrak p}$ for all prime ideals ${\mathfrak p}$ lying over $p$. 
\end{definition}

\subsection{Families and communities}
\label{sec:106}
\begin{definition} 
\label{defn:109}
A Hida community is the set $\left\{{\mathcal F} \right\}$ of primitive $p$-ordinary cuspidal Hida families having the same tame level and the same residual representation. 
\end{definition}
We note that a Hida community is always a finite set. By definition, it makes sense to speak of the residual representation and $p$-distinguishability of a Hida community.  

\section{Weight one specializations of a Hida family}
\label{sec:110}
\subsection{CM and non-CM families}
\label{sec:111}
\begin{definition} 
\label{defn:105}
A primitive $p$-ordinary cuspidal Hida family ${\mathcal F}$ is of CM type (or a CM family) if there exists a totally imaginary quadratic extension $K$ of $F$ such that $\rho_{\mathcal F} \cong \rho_{\mathcal F} \otimes \varepsilon_{K/F}$, where $\varepsilon_{K/F}: G_F \rightarrow \left\{\pm 1 \right\}$ is the quadratic character corresponding to $K/F$. We say that ${\mathcal F}$ is a non-CM family if ${\mathcal F}$ is not of CM type. 
\end{definition}
A CM family is explicitly constructed by a $\Lambda$-adic Hecke character of $K$. We refer the reader to Section 4 of \cite{BGV} for the construction. We note that a CM family has infinitely many classical specializations including in weight one, and all of them have CM in the classical sense. 

Hida's control theorem (Theorem 3; p.~298 of \cite{H88}) tells us that if $f$ is a $p$-ordinary $p$-stabilized newform of weight $k \geq 2$, then there is a unique primitive $p$-ordinary cuspidal Hida family ${\mathcal F}$ specializing to $f$. In addition, if $f$ has CM in the classical sense, one can explicitly construct a CM family passing through $f$, and by the uniqueness this CM family is the only one specializing to $f$. Therefore, any arithmetic specialization of a non-CM family does not have CM in the classical sense. 

\begin{remark}
\label{rmk:106}
If $f$ is of weight $1$, Wiles showed the existence of a primitive $p$-ordinary cuspidal Hida family ${\mathcal F}$ specializing to $f$ as well (Theorem 3; p.~532 of \cite{W88}), but the uniqueness of ${\mathcal F}$ is not guaranteed. 
\end{remark}

\subsection{Finiteness result for non-CM families}
\label{sec:112}
By definition, a primitive $p$-ordinary cuspidal Hida family ${\mathcal F}$ admits infinitely many classical specializations of weight at least two. If ${\mathcal F}$ is of CM type, then ${\mathcal F}$ contains infinitely many classical weight one specializations as well. The proof of Balasubramanyam, Ghate and Vatsal's theorem implies the following finiteness result. 
\begin{theorem}[cf. Theorem 3; p.~517 of \cite{BGV}]
\label{theo:107}
A primitive $p$-ordinary cuspidal Hida family ${\mathcal F}$ admits infinitely many classical weight one specializations if and only if ${\mathcal F}$ is of CM type. 
\end{theorem}
\begin{remark}
\label{rmk:108}
\begin{enumerate}
\item It should be pointed out that in the original paper \cite{BGV}, the Iwasawa algebra $\Lambda$ is isomorphic to a power series ring of $(1+\delta)$-variable, where $\delta$ is the Leopoldt defect for $F$ and $p$. In this paper we follow Wiles' setting in \cite{W88} and consider only the cyclotomic variable. Hence our Iwasawa algebra $\Lambda$ is isomorphic to a power series ring of one variable ${\mathcal O}[[X]]$ and we can apply Weierstrass' $p$-adic preparation theorem to $\Lambda$. In particular we do not need Lemma 1 of \cite{BGV}. This makes our argument much simpler, as seen in the proof of Theorem \ref{theo:107} below. 
\item In the original theorem in \cite{BGV}, the assumptions are
\begin{itemize}
\item[(A1)] $p$ splits completely in $F$; and
\item[(A2)] the family ${\mathcal F}$ is $p$-distinguished and $\overbar{\rho}_{\mathcal F}$ is absolutely irreducible when restricted to ${\rm Gal}(\bar{F}/F(\zeta_p))$, where $\zeta_p$ is a primitive $p$-th root of unity, 
\end{itemize}
and the statement is
\begin{itemize}
\item[(C1)] $\rho_{\mathcal F}|_{D_{\mathfrak p}}$ splits for each ${\mathfrak p} \mid p$; if and only if
\item[(C2)] ${\mathcal F}$ is of CM type. 
\end{itemize} 
They had to impose the assumptions (A1) and (A2) because they used modularity lifting theorem (Theorem 3; p.~999 of \cite{Sa10}) to observe that (C1) implies any weight one specializations of ${\mathcal F}$ is classical. More precisely, (A1) is a condition that is assumed in the modularity lifting theorem (see the remark in p.~518 of \cite{BGV}), and (A2) is used to verify that they can apply modularity lifting theorem to the Galois representation attached to each weight one specialization of ${\mathcal F}$. We do not need modularity lifting theorem to establish the equivalence in Theorem \ref{theo:107}, and hence we assume neither (A1) nor (A2). 
\end{enumerate}
\end{remark}
\begin{proof}
It follows from the construction that a CM family admits infinitely many classical weight one specializations. We prove the converse. As explained in the introduction, the Galois representation attached to a classical weight one form is either dihedral or exceptional. By the same reasoning as in \cite{GV} p.~2155, we see that only finitely many classical weight one specializations of ${\mathcal F}$ are exceptional. Therefore ${\mathcal F}$ has infinitely many classical weight one specializations $f$ such that the associated Galois representation $\rho_f$ satisfies $\rho_f \cong \varepsilon_{K/F} \otimes \rho_f$ for some quadratic extension $K$ of $F$. Since the conductor of $\rho_f$ is ${\mathfrak n}_0p^{r+1}$ for some integer $r \geq 0$, we see that there is a bound on the discriminant of $K/F$. Hence there exists a quadratic extension $K$ of $F$ such that infinitely many classical weight one specializations $f$ of ${\mathcal F}$ satisfy $\rho_f \cong \varepsilon_{K/F} \otimes \rho_f$. This implies that ${\rm Tr}\rho_f({\rm Frob}_{\mathfrak l})=0$ for all primes ${\mathfrak l}$ of $F$ prime to ${\mathfrak n}_0p$ and inert in $K$. Since the intersection of infinitely many height one prime ideals of $\Lambda$ is zero, we have ${\rm Tr}\rho_{\mathcal F}({\rm Frob}_{\mathfrak l})=0$ for all such primes ${\mathfrak l}$. As for a prime ideal ${\mathfrak l}$ that splits in $K$, the equality ${\rm Tr}\rho_{\mathcal F}({\rm Frob}_{\mathfrak l})=\varepsilon_{K/F}({\rm Frob}_{\mathfrak l}){\rm Tr}\rho_{\mathcal F}({\rm Frob}_{\mathfrak l})$ is unconditional. Hence $\rho_{\mathcal F}$ and $\varepsilon_{K/F} \otimes \rho_{\mathcal F}$ have the same trace. As $\rho_{\mathcal F}$ is irreducible and hence determined by the trace, this implies $\rho_{\mathcal F} \cong \varepsilon_{K/F} \otimes \rho_{\mathcal F}$. If $K/F$ is not totally imaginary, this yields a contradiction, since the automorphic representation associated to any specializations of ${\mathcal F}$ at arithmetic points of weight $k \geq 2$ is a holomorphic discrete series. Therefore $K/F$ is totally imaginary and the family ${\mathcal F}$ has CM by $K$. 
\end{proof}
The goal of this paper is to make this finiteness result effective. 

\subsection{Projective image of the residual Galois representation}
\label{sec:113}
We recall that the residue field ${\mathbb{F}}$ of $\Lambda_L$ is a finite field of characteristic $p$. By Dickson's classification of subgroups of $PGL_2({\mathbb{F}})$, the image of $\overbar{\rho}_{\mathcal F}^{\rm ss}: G_F \rightarrow GL_2({\mathbb{F}})$ in $PGL_2({\mathbb{F}})$ is either
\begin{enumerate}
\item a subgroup of a Borel subgroup in $PGL_2({\mathbb{F}}')$ with a quadratic extension ${\mathbb{F}}'/{\mathbb{F}}$; 
\item conjugate to $PGL_2({\mathbb{F}}')$ or $PSL_2({\mathbb{F}}')$ for some subfield ${\mathbb{F}}'$ of ${\mathbb{F}}$; 
\item isomorphic to the dihedral group $D_{2m}$ of order $2m$ with $m$ prime to $p$, or the symmetric group $S_4$, the alternative groups $A_4$ or $A_5$. 
\end{enumerate}
The residual representation $\bar{\rho}_{\mathcal F}^{\rm ss}$ is reducible if (1) occurs and irreducible if (2) or (3) is the case. On the other hand, it is well known that the Galois representation $\rho_f$ attached to a classical weight one cuspidal eigenform $f$ is an irreducible totally odd Artin representation and has finite image. Let us write $\rho_f: G_F \rightarrow GL_2(O)$ for the $p$-adic integer ring $O$ of a finite extension of ${\mathbb{Q}}_p$. Then the image of $\rho_f$ in $PGL_2(O)$ is isomorphic to either $D_{2n}$ for some integer $n \geq 1$ or $A_4, S_4, A_5$. The weight one form $f$ is said to be of dihedral type in the first case and of exceptional type in the latter three cases. 
\begin{lemma}
\label{lem:110}
Suppose that a Hida family ${\mathcal F}$ has a classical weight one specialization $f$. 
\begin{enumerate}
\item If $f$ is of dihedral type and the image of $\rho_f$ in $PGL_2(O)$ is isomorphic to $D_{2n}$, then the image of $\overbar{\rho}_{\mathcal F}^{\rm ss}$ in $PGL_2({\mathbb{F}})$ is isomorphic to $D_{2m}$, where $m$ is the prime-to-$p$ part of $n$. 
\item If $f$ is of exceptional type, the image of $\overbar{\rho}_{\mathcal F}^{\rm ss}$ in $PGL_2({\mathbb{F}})$ is isomorphic to that of $\rho_f$ in $PGL_2(O)$. 
\end{enumerate}
\end{lemma}
This lemma can be shown in the same way as Lemma 4.2; p.~674 of \cite{DG}. For the purpose of giving an upper bound on the number of classical weight one specializations, it is enough to consider families as in Lemma \ref{lem:110}. 
\begin{definition}
\label{defn:111}
A Hida family ${\mathcal F}$ is residually of dihedral (resp. exceptional) type if the image of $\overbar{\rho}_{\mathcal F}^{\rm ss}: G_F \rightarrow GL_2({\mathbb{F}})$ in $PGL_2({\mathbb{F}})$ is isomorphic to $D_{2m}$ (resp. $A_4, S_4$ or $A_5$). 
\end{definition}
Notice that if ${\mathcal F}$ is residually of dihedral or exceptional type, then $\bar{\rho}_{\mathcal F}=\bar{\rho}_{\mathcal F}^{\rm ss}$ is irreducible and unique up to isomorphism. We also note that if ${\mathcal F}$ is residually of dihedral type, then any classical weight one specialization of ${\mathcal F}$ (if exists) has to be of dihedral type. The same for each of $A_4, S_4$ and $A_5$ type (cf. Lemma 4.5; p.~675 of \cite{DG}). 

\section{Families residually of dihedral type}
\label{sec:200}
\subsection{Setting and the main theorem}
\label{sec:201}
We consider a primitive $p$-ordinary cuspidal Hida family ${\mathcal F}$ of tame level ${\mathfrak n}_0$ and nebentype $\psi_{\mathcal F}$ satisfying the following properties: 
\begin{itemize}
\item[(P1)] ${\mathcal F}$ is residually of dihedral type: say, there exists a quadratic extension $K$ of $F$ such that $\bar{\rho}_{\mathcal F}$ is equivalent to the induced representation ${\rm Ind}_K^F(\bar{\varphi})$ by a character $\bar{\varphi}: G_K={\rm Gal}(\bar{F}/K) \rightarrow {\mathbb{F}}^\times$; 
\item[(P2)] ${\mathcal F}$ has a classical weight one specialization $f$ such that the associated representation $\rho_f: G_F \rightarrow GL_2(O)$ ($O$ is the $p$-adic integer ring of a suitable finite extension of ${\mathbb{Q}}_p$) satisfies: $\rho_f(I_{\mathfrak p})$ has order at least three for each prime ${\mathfrak p}$ of $F$ lying over $p$. 
\end{itemize}
\begin{remark}
\label{rmk:201}
The first property (P1) implies that $\bar{\rho}_{\mathcal F}$ is irreducible, and hence $\bar{\varphi} \neq \bar{\varphi}^\sigma$. Here $\sigma \in G_F$ is a generator of ${\rm Gal}(K/F)$ and $\bar{\varphi}^\sigma$ is defined by $\bar{\varphi}^\sigma(g)=\bar{\varphi}(\sigma g \sigma^{-1})$ for each $g \in G_K$ (this is Mackey's criterion for irreducibility of induced representations). 
\end{remark}
In order to state the main theorem, we need two lemmas: 
\begin{lemma}
\label{rmk:221}
The representation $\rho_f$ in the second property {\rm (P2)} is of the form ${\rm Ind}_K^F(\varphi)$ for a finite order character $\varphi: G_K \rightarrow O^\times$ which is a lift of $\bar{\varphi}$. 
\end{lemma}
\begin{proof}
With Lemma \ref{lem:110} in mind, the image of $\rho_f$ in $PGL_2(O)$ is $D_{2n}$ for an integer $n \geq 1$. Let $C_n$ be a cyclic subgroup of order $n$ in $D_{2n}$. If $n \leq 2$, such a subgroup is not uniquely determined, but we make the choice so that the image of $C_n \subset PGL_2(O)$ in $PGL_2({\mathbb{F}})$ is $\bar{\rho}_{\mathcal F}(G_K)$. Then we see that the image of $\rho_f(g)$ in $PGL_2(O)$ is contained in $C_n$ if and only if $g \in G_K$. This implies that ${\rm Tr}(\rho_f)(g)=0$ if $g \notin G_K$, and hence we have ${\rm Tr}(\rho_f)={\rm Tr}(\rho_f \otimes \varepsilon_{K/F})$ (recall that $\varepsilon_{K/F}: G_F \rightarrow \left\{\pm 1 \right\}$ is the quadratic character corresponding to $K/F$). Since $\rho_f$ is irreducible, we conclude that $\rho_f \cong \rho_f \otimes \varepsilon_{K/F}$. 

By definition there exists a matrix $M \in GL_2(O)$ such that $M\rho_f(g)M^{-1}=\varepsilon_{K/F}(g)\rho_f(g)$ for all $g \in G_F$. Since $\varepsilon_{K/F}^2=1$, we have
\begin{alignat*}{3}
 M^2\rho_f(g)M^{-2} &= \varepsilon_{K/F}(g)^2\rho_f(g) &= \rho_f(g). 
\end{alignat*}
Irreducibility of $\rho_f$ implies that $M^2$ is a scalar matrix. We change the basis of $\rho_f$ if necessary and assume that $M=\left(
\begin{smallmatrix}
 \alpha & \beta \\
 0 & \delta
\end{smallmatrix}
\right)$. Let $\rho_f(\sigma)=\left(
\begin{smallmatrix}
 a & b \\
 c & d
\end{smallmatrix}
\right)$ for our fixed generator $\sigma \in G_F$ of ${\rm Gal}(K/F)$. Then we have $M\rho_f(\sigma)M^{-1}=-\rho_f(\sigma)$ and thus
$$2\alpha a=-\beta c, \ b(\alpha +\delta)+\beta(a+d)=0, \ c(\alpha+\delta)=0, \ 2\delta d=-\beta c. $$
If $\alpha+\delta \neq 0$ then $c=0$ and $a=d=0$, a contradiction. Therefore $\alpha=-\delta$ and $M$ has two distinct eigenvalues $\alpha$ and $-\alpha$. Again we change the basis of $\rho_f$ and assume that $M=\left(
\begin{smallmatrix}
 \alpha & 0 \\
 0 & -\alpha
\end{smallmatrix}
\right)$. Let $V$ denote the representation space of $\rho_f$. Then $M$ acts semi-simply on $V$. Let $V(\alpha)$ (resp. $V(-\alpha)$) be the $\alpha$-eigenspace (resp. $(-\alpha)$-eigenspace) of $M$. If $v \in V(\pm \alpha)$ then $M\rho_f(g)v=\pm \alpha \varepsilon_{K/F}(g)\rho_f(g)v$. Therefore if $g \in G_K$ then $\rho_f(g)$ preserves each of $V(\pm \alpha)$, and if $g \notin G_F$ then $\rho_f(g)$ permutes $V(\alpha)$ and $V(-\alpha)$. Now let $\varphi: G_K \rightarrow O^\times$ be the character defined by $\rho_f(g)v=\varphi(g)v$ for all $v \in V(\alpha)$. Then a direct calculation shows that $\rho_f$ is induced by $\varphi$.  
\end{proof}
\begin{lemma} 
\label{lem:202}
Each prime ${\mathfrak p}$ of $F$ lying over $p$ splits in $K$. 
\end{lemma}
\begin{proof}
Let $\rho_f={\rm Ind}_K^F(\varphi)$ as in Lemma \ref{rmk:221} and $D_{2n}$ the projective image of $\rho_f$. Since ${\mathcal F}$ is $p$-ordinary, so is $f$ and we have
\begin{align}
 \rho_f|_{I_{\mathfrak p}} & \cong \left(
\begin{array}{cc}
 \psi_{\mathcal F} \varepsilon & * \\
 0 & 1
\end{array}
\right) \label{eqn:204}
\end{align}
(here $\varepsilon: {\mathbf{G}} \rightarrow \overbar{\mathbb{Q}}_p^\times$ is the character such that $\varphi_{1, \varepsilon}({\mathcal F})=f$). Therefore $\rho_f(I_{\mathfrak p})$ is a finite cyclic group and it injects into $PGL_2(O)$. By our assumption that $\rho_f(I_{\mathfrak p})$ has order at least three, we know that the integer $n$ is at least three and the projective image of $\rho_f(I_{\mathfrak p})$ is contained in the unique cyclic group $C_n$ of order $n$ in $D_{2n}$. Therefore $I_{\mathfrak p}$ is contained in $G_K$, that is, ${\mathfrak p}$ is unramified in $K$. This observation implies
\begin{align}
 \rho_f|_{I_{\mathfrak p}} & \cong \left(
\begin{array}{cc}
 \varphi & 0 \\
 0 & \varphi^\sigma 
\end{array}
\right) \quad (\varphi^\sigma(g)=\varphi(\sigma g \sigma^{-1})). \label{eqn:203}
\end{align}
Since the representations (\ref{eqn:204}) and (\ref{eqn:203}) are equivalent, we know that at least one of $\varphi$ and $\varphi^\sigma$ has to be unramified at ${\mathfrak p}$. Suppose that both $\varphi$ and $\varphi^\sigma$ are unramified at ${\mathfrak p}$. We know conductor-discriminant formula
\begin{align}
 {\rm cond}(\rho_f) &= N_{K/F}({\rm cond}(\varphi))d_{K/F}, \label{eqn:c-d}
\end{align}
where $d_{K/F}$ (resp. $N_{K/F}$) is the relative discriminant (resp. the relative ideal norm) of $K/F$. Since the left-hand side of (\ref{eqn:c-d}) is divisible by ${\mathfrak p}$, and ${\mathfrak p}$ is unramified in $K$, $N_{K/F}({\rm cond}(\varphi))$ has to be divisible by ${\mathfrak p}$. This contradicts our assumption. Therefore exactly one of $\varphi$ and $\varphi^\sigma$ is ramified at ${\mathfrak p}$. This cannot be satisfied if ${\mathfrak p}$ is inert in $K$. 
\end{proof}
For the prime ideals ${\mathfrak p}_1, \ldots, {\mathfrak p}_t$ of $F$ lying over $p$, let ${\mathfrak p}_iO_K={\mathcal P}_i{\mathcal P}_i^\sigma$, $\varphi$ is ramified at ${\mathcal P}_i$ and unramified at ${\mathcal P}_i^\sigma$ for each $i=1, \ldots, t$. Further we put ${\mathcal Q}=\prod_{i=1}^t {\mathcal P}_i$ and $U_K=\left\{u \in O_K^\times \mid u \equiv 1 \bmod {\mathcal Q} \right\}$. For each $i=1, \ldots, t$, $U_{K, {\mathcal P}_i}$ denotes the principal unit group of the completion of $K$ at ${\mathcal P}_i$. $\overbar{U_K}$ is the closure (taken in $\prod_{i=1}^t U_{K, {\mathcal P}_i}$) of the image of $U_K$ under the diagonal map $U_K \rightarrow \prod_{i=1}^t U_{K, {\mathcal P}_i}$. Let ${\rm Cl}_K({\mathfrak n}_0{\mathcal Q}^r)$ be the narrow ray class group of $K$ of modulus ${\mathfrak n}_0{\mathcal Q}^r$ for each integer $r \geq 1$. Our main theorem is related to the finiteness of the projective limit ${\rm Cl}_K({\mathfrak n}_0{\mathcal Q}^\infty)=\varprojlim_r {\rm Cl}_K({\mathfrak n}_0{\mathcal Q}^r)$. In view of class field theory, ${\rm Cl}_K({\mathfrak n}_0{\mathcal Q}^\infty)$ is a finite group if and only if $(\prod_{i=1}^t U_{K, {\mathcal P}_i})/\overbar{U_K}$ is a finite group (see Section \ref{sec:202} for detail). If this is the case, we put
\begin{align*}
 M' &= |{\rm Cl}_K| \cdot \left|\left(\prod_{i=1}^t U_{K, {\mathcal P}_i} \right)/\overbar{U_K} \right| \cdot \prod_{\substack{{\mathfrak l} \mid {\mathfrak n}_0, \\ \text{split in $K$}}} (q_{\mathfrak l}-1) \cdot \prod_{\substack{{\mathfrak l} \mid {\mathfrak n}_0, \\ \text{inert in $K$}}} (q_{\mathfrak l}+1) 
\end{align*}
and $M({\mathcal F}, K, f)=p^{{\rm ord}_p(M')}$ the $p$-part of $M'$. Here ${\rm ord}_p$ is the $p$-adic valuation normalized so that ${\rm ord}_p(p)=1$, $|{\rm Cl}_K|$ is the class number of $K$, and for a prime ${\mathfrak l}$ of $F$, $q_{\mathfrak l}$ is the order of the residue field $O_F/{\mathfrak l}$. 
\begin{theorem}
\label{theo:205}
Let $p$ be an odd prime. Suppose that a primitive $p$-ordinary cuspidal Hida family ${\mathcal F}$ of tame level ${\mathfrak n}_0$ prime to $p$ satisfies the properties {\rm (P1)} and {\rm (P2)} above. Then the following two statements hold true:  
\begin{enumerate}
\item If ${\rm Cl}_K({\mathfrak n}_0{\mathcal Q}^\infty)$ is an infinite group, then $K/F$ is totally imaginary and there is a Hida family $\mathcal G$ of CM type by $K$ that belongs to same Hida community as ${\mathcal F}$; 
\item If ${\rm Cl}_K({\mathfrak n}_0{\mathcal Q}^\infty)$ is of finite order, then ${\mathcal F}$ is not a CM family and the number of classical weight one specializations of ${\mathcal F}$ is bounded by $M({\mathcal F}, K, f)$. 
\end{enumerate}
\end{theorem}
\begin{remark}
\label{rmk:291}
\begin{enumerate}
\item When $K/F$ is totally imaginary, the rank of the torsion-free part of $U_K$ as a ${\mathbb{Z}}$-module is $d-1$, and hence the ${\mathbb{Z}}_p$-rank of $\overbar{U_K}$ is at most $d-1$. On the other hand, the ${\mathbb{Z}}_p$-rank of $\prod_{i=1}^t U_{K, {\mathcal P}_i}$ is $d$. Therefore the group $\left(\prod_{i=1}^t U_{K, {\mathcal P}_i} \right)/\overbar{U_K}$ cannot be finite. 
\item If $K/F$ is totally imaginary and ${\mathcal F}$ is not a CM family, one can also give an upper bound. This estimate will be discussed at the end of this section (Proposition \ref{prop:294}), exactly in the same manner as Lemma 6.5; p.~682 of \cite{DG}. \label{rmk:293}
\end{enumerate}
\end{remark}
The outline of the proof is as follows. Suppose that $f'$ is another classical weight one specialization of ${\mathcal F}$ (if any) and let $\rho_{f'}={\rm Ind}_K^F(\varphi')$. Our strategy is counting the number of characters $\xi=\varphi/\varphi'$. Let $G_F^{\rm ab}$ (resp. $G_K^{\rm ab}$) be the maximal abelian quotient of $G_F$ (resp. $G_K$) and ${\rm ver}_{K/F}: G_F^{\rm ab} \rightarrow G_K^{\rm ab}$ the transfer map (namely, ${\rm ver}_{K/F}(g)=g\sigma g \sigma^{-1}$ if $g \in G_K$ and ${\rm ver}_{K/F}(g)=g^2$ if $g \notin G_K$). 
\begin{lemma}
\label{lem:293}
The above $\xi: G_K \rightarrow \overbar{\mathbb{Q}}_p^\times$ satisfies the following properties: 
\begin{enumerate}
\item $\xi$ is a $p$-power order character; 
\item $\xi$ is unramified outside ${\mathfrak n}_0{\mathcal Q}$ and the infinite places of $K$, and the prime-to-$p$ part of $N_{K/F}({\rm cond}(\xi))$ divides ${\mathfrak n}_0$; 
\item $\xi \circ {\rm ver}_{K/F}: G_F \rightarrow \overbar{\mathbb{Q}}_p^\times$ is unramified outside $p$. 
\end{enumerate}
\end{lemma}
\begin{proof}
Since both $\varphi$ and $\varphi'$ are lifts of $\bar{\varphi}$, they differ by a $p$-power order character. As for the second statement, by an argument similar to the proof of Lemma \ref{lem:202}, we may assume that $\varphi'$ is ramified at ${\mathcal P}_i$'s and unramified at ${\mathcal P}_i^\sigma$'s. Then $\xi$ is unramified at ${\mathcal P}_i^\sigma$ for all $i$. We note that the conductor of $\rho_f$ is equal to the level of $f$ and hence it is of the form ${\mathfrak n}_0p^r$ for some integer $r$. Then conductor-discriminant formula implies that the prime-to-$p$ part of $N_{K/F}({\rm cond}(\varphi))$ divides ${\mathfrak n}_0$ (the same for $\varphi'$). Therefore the prime-to-$p$ part of $N_{K/F}({\rm cond}(\xi))$ divides ${\mathfrak n}_0$, and $\xi$ is unramified outside ${\mathfrak n}_0{\mathcal Q}$ and the infinite places. We prove the third assertion. We have ${\rm det}(\rho_f)=\varepsilon_{K/F} \cdot (\varphi \circ {\rm ver}_{K/F})$ (the same for $f'$ and $\varphi'$) and thus $\xi \circ {\rm ver}_{K/F}={\rm det}(\rho_f)/{\rm det}(\rho_{f'})=\varepsilon/\varepsilon'$ (here $\varepsilon: {\mathbf{G}} \rightarrow \overbar{\mathbb{Q}}_p^\times$ is the character such that $f=\varphi_{1, \varepsilon}({\mathcal F})$. The same for $f'$ and $\varepsilon'$). Therefore $\xi \circ {\rm ver}_{K/F}$ is unramified outside $p$ and the infinite places. Moreover ${\rm det}(\rho_f)$ and ${\rm det}(\rho_{f'})$ are both totally odd, since $f$ and $f'$ are weight one forms. Hence $\xi \circ {\rm ver}_{K/F}$ is totally even, as it is desired. 
\end{proof}
Hence $\xi$ can be regarded as a $p$-power order character of ${\rm Cl}_K({\mathfrak n}_0{\mathcal Q}^\infty)$. The restriction of such a character to the product $\prod_{i=1}^t U_{K, {\mathcal P}_i} \times (O_K/{\mathfrak n}_0)^\times$ factors through the quotient (\ref{eqn:207}) below (Section \ref{sec:202}). Hence we obtain the upper bound in Theorem \ref{theo:205} (2) if ${\rm Cl}_K({\mathfrak n}_0{\mathcal Q}^\infty)$ is finite. Secondly we assume that ${\rm Cl}_K({\mathfrak n}_0{\mathcal Q}^\infty)$ is infinite. We first show that there are infinitely many classical weight one specializations that arise from the Hida community containing ${\mathcal F}$ (Section \ref{sec:203}). Since a Hida community is a finite set, there exists a member ${\mathcal G}$ in the community which possesses infinitely many such specializations. In view of Theorem \ref{theo:107}, $K/F$ has to be totally imaginary and ${\mathcal G}$ has CM by $K$ (Section \ref{sec:204}). 

\subsection{Global class field theory}
\label{sec:202}
We begin with some observation on the narrow ray class group ${\rm Cl}_K({\mathfrak n}_0{\mathcal Q}^\infty)$. For each $i=1, \ldots, t$, let $K_{{\mathcal P}_i}$ be the completion of $K$ at ${\mathcal P}_i$, $O_{K, {\mathcal P}_i}$ the integer ring of $K_{{\mathcal P}_i}$, $O_K^\times$ be the unit group of $K$ and $\overbar{O_K^\times}$ be the closure (taken in $\prod_{i=1}^t O_{K, {\mathcal P}_i}^\times$) of the image of $O_K^\times$ under the diagonal embedding $O_K^\times \rightarrow \prod_{i=1}^t O_{K, {\mathcal P}_i}^\times$. By global class field theory, we know that ${\rm Cl}_K({\mathfrak n}_0{\mathcal Q}^\infty)$ is sitting in the exact sequence: 
\begin{align}
\begin{CD}
 1 \; @>>> \; \overbar{O_K^\times} \; @>>> \; \displaystyle \prod_{i=1}^t O_{K, {\mathcal P}_i}^\times \times (O_K/{\mathfrak n}_0)^\times \; @>>> \; {\rm Cl}_K({\mathfrak n}_0{\mathcal Q}^\infty) \; @>>> \; {\rm Cl}_K^+ \; @>>> \; 1. 
\end{CD} \label{exact}
\end{align}
Here ${\rm Cl}_K^+$ is the narrow ideal class group of $K$. Notice that the $p$-part of $|{\rm Cl}_K^+|$ and $|{\rm Cl}_K|$ are equal, since $p$ is odd. In particular we see that
\begin{center}
 ``${\rm Cl}_K({\mathfrak n}_0{\mathcal Q}^\infty)$ is a finite group if and only if the ${\mathbb{Z}}_p$-rank of $\overbar{O_K^\times}$ is equal to $d$." 
\end{center}
Let $\xi: {\rm Cl}_K({\mathfrak n}_0{\mathcal Q}^\infty) \rightarrow \overbar{\mathbb{Q}}_p^\times$ be a $p$-power order character. Then $\xi$ restricted to the product $\prod_{i=1}^t O_{K, {\mathcal P}_i}^\times \times (O_K/{\mathfrak n}_0)^\times$ factors through the quotient
\begin{align}
 \left(\left(\prod_{i=1}^t U_{K, {\mathcal P}_i} \right)/\overbar{U_K} \right) \times (O_K/{\mathfrak n}_0)^\times. \label{eqn:206}
\end{align}
Suppose further that $\xi \circ {\rm ver}_{K/F}: G_F^{\rm ab} \rightarrow \overbar{\mathbb{Q}}_p^\times$ is unramified outside $p$. Note that $\xi \circ {\rm ver}_{K/F}$ is unramified outside ${\mathfrak n}_0p$ because $\xi$ is unramified outside ${\mathfrak n}_0{\mathcal Q}$. However it is not always the case that $\xi$ is unramified at primes dividing ${\mathfrak n}_0$. Therefore this additional assertion should cause some constraint on the behavior of $\xi$ at the primes of $K$ lying over ${\mathfrak n}_0$. We shall establish the following
\begin{lemma}
\label{lem:tame-level}
Let $\xi$ be a character satisfying the properties $(1)$ through $(3)$ in Lemma $\ref{lem:293}$. Then $\xi$ restricted to $\prod_{i=1}^t O_{K, {\mathcal P}_i}^\times \times (O_K/{\mathfrak n}_0)^\times$ factors through the quotient 
\begin{align}
 \left(\left(\prod_{i=1}^t U_{K, {\mathcal P}_i} \right)/\overbar{U_K} \right) \times \prod_{\substack{{\mathfrak l} \mid {\mathfrak n}_0, \\ \text{split in $K$}}} {\mathbb{F}}_{\mathfrak l}^\times \times \prod_{\substack{{\mathfrak l} \mid {\mathfrak n}_0, \\ \text{inert in $K$}}} {\mathbb{F}}_{{\mathfrak l}^2}^\times/{\mathbb{F}}_{\mathfrak l}^\times \label{eqn:207}
\end{align}
of $(\ref{eqn:206})$. Here for a prime ideal ${\mathfrak l}$ of $F$ dividing ${\mathfrak n}_0$, ${\mathbb{F}}_{\mathfrak l}$ denotes the residue field $O_F/{\mathfrak l}$. As for a prime factor ${\mathfrak l}$ of ${\mathfrak n}_0$ that is inert in $K$, ${\mathbb{F}}_{{\mathfrak l}^2}$ is the residue field $O_K/{\mathfrak l}O_K$, which is the unique quadratic extension of ${\mathbb{F}}_{\mathfrak l}$. 
\end{lemma}
\begin{proof}
Let ${\mathfrak l}$ be a prime ideal dividing ${\mathfrak n}_0$ and $O_{F, {\mathfrak l}}^\times$ the unit group of the completion of $F$ at ${\mathfrak l}$. Similarly for a prime ideal ${\mathfrak L}$ of $K$, $O_{K, {\mathfrak L}}^\times$ denotes the unit group of the completion of $K$ at ${\mathfrak L}$. We note that in view of class field theory, the transfer map ${\rm ver}_{K/F}: G_F^{\rm ab} \rightarrow G_K^{\rm ab}$ corresponds to a canonical map from the id\'ele class group of $F$ to that of $K$. With this in mind, what we have to do is to characterize the set of $p$-power order characters $\xi: \prod_{{\mathfrak L} \mid {\mathfrak l}} O_{K, {\mathfrak L}}^\times \rightarrow \overbar{\mathbb{Q}}_p^\times$ whose restriction to $O_{F, {\mathfrak l}}^\times$ is trivial. Since ${\mathfrak n}_0$ is prime to $p$, it is enough to investigate the quotient group
\begin{align}
 \left(\prod_{{\mathfrak L} \mid {\mathfrak l}} (O_{K, {\mathfrak L}}/{\mathfrak L})^\times \right)/(O_{F, {\mathfrak l}}/{\mathfrak l})^\times \label{eqn:quot}
\end{align}

Suppose first that ${\mathfrak l}$ splits in $K$, say ${\mathfrak l}O_K={\mathfrak L}{\mathfrak L}^\sigma$ for distinct prime ideals ${\mathfrak L}$ and ${\mathfrak L}^\sigma$ of $K$. Then $\sigma$ induces an isomorphism $\sigma: O_{K, {\mathfrak L}}^\times \cong O_{K, {\mathfrak L}^\sigma}^\times$ and the group (\ref{eqn:quot}) is isomorphic to $(O_{K, {\mathfrak L}^\sigma}/{\mathfrak L}^\sigma)^\times$ by the map taking $(x, y) \bmod (O_{F, {\mathfrak l}}/{\mathfrak l})^\times$ to $y \cdot (x^\sigma)^{-1}$. Secondly if ${\mathfrak l}$ is inert in $K$, the group (\ref{eqn:quot}) is nothing but ${\mathbb{F}}_{{\mathfrak l}^2}^\times/{\mathbb{F}}_{\mathfrak l}^\times$, as it is desired. Finally the group (\ref{eqn:quot}) is trivial if ${\mathfrak l}$ is ramified in $K$. 
\end{proof}
Now we are ready to prove Theorem \ref{theo:205} (2). If ${\rm Cl}_K({\mathfrak n}_0{\mathcal Q}^\infty)$ is a finite group, so is the group (\ref{eqn:207}) and the number of classical weight one specializations of ${\mathcal F}$ is bounded by the order of the $p$-Sylow group of (\ref{eqn:207}). This establishes Theorem \ref{theo:205} (2). 

Secondly we shall prove Theorem \ref{theo:205} (1). Put ${\rm Cl}_K({\mathcal Q}^\infty)=\varprojlim_r {\rm Cl}_K({\mathcal Q}^r)$. In view of the exact sequence (\ref{exact}), ${\rm Cl}_K({\mathcal Q}^\infty)$ is infinite if and only if ${\rm Cl}_K({\mathfrak n}_0{\mathcal Q}^\infty)$ is infinite. A character $\xi: {\rm Cl}_K({\mathcal Q}^\infty) \rightarrow \overbar{\mathbb{Q}}_p^\times$, regarded as a character of $G_K$, is unramified outside ${\mathcal Q}$ and the infinite places. We record this observation as 
\begin{lemma}
\label{lem:208}
Suppose that ${\rm Cl}_K({\mathfrak n}_0{\mathcal Q}^\infty)$ is an infinite group. Then there are infinitely many characters $\xi: G_K \rightarrow \overbar{\mathbb{Q}}_p^\times$ which enjoy the properties $(1)$ through $(3)$ in Lemma $\ref{lem:293}$ and is unramified at the prime ideals of $K$ lying over ${\mathfrak n}_0$. 
\end{lemma} 

\subsection{Characters and classical weight one forms}
\label{sec:203}
In this subsection, we attach a classical weight one cuspidal eigenform to each character $\xi$ treated in the previous section. 
\begin{lemma}
\label{lem:292}
Let $\xi$ be a character satisfying the properties $(1)$ through $(3)$ in Lemma $\ref{lem:293}$ and put $\rho_\xi={\rm Ind}_K^F(\varphi \xi)$. Then $\rho_\xi$ is an irreducible, totally odd representation and is unramified outside ${\mathfrak n}_0pd_{K/F}$. 
\end{lemma}
\begin{proof}
Since ${\mathcal F}$ is residually of dihedral type, the reduction $\bar{\rho}_\xi=\bar{\rho}_{\mathcal F}$ is irreducible, and so is $\rho_\xi$. We have ${\rm det}(\rho_\xi)=\varepsilon_{K/F} \cdot ((\varphi \xi) \circ {\rm ver}_{K/F})$ as a character of $G_F^{\rm ab}$ since $\rho_\xi$ is induced by $\varphi \xi$. As $p$ is odd and $\xi$ is a $p$-power order character, we have $(\xi \circ {\rm ver}_{K/F})(c)=1$ for any complex conjugate $c$ in $G_F$. Therefore the parity $\rho_\xi$ and $\rho_f$ are the same and thus $\rho_\xi$ is totally odd. The last assertion follows from conductor-discriminant formula (\ref{eqn:c-d}). 
\end{proof}
\begin{lemma}
\label{lem:wt-one}
Let $\rho_\xi$ be the induced representation in Lemma $\ref{lem:292}$. 
There exists a Hilbert modular form $f_\xi$ over $F$ of parallel weight one so that $\rho_{f_\xi} \cong \rho_\xi$.  
\end{lemma}
\begin{proof}
Let $\tau_\xi$ be the automorphic form on $GL_1(\mathbb{A}_F)$ associated to $\xi$ via class field theory. Here ${\mathbb{A}}_F$ is the ring of ad\`eles of $F$. Let $\pi={\rm AI}^F_K(\tau_\xi)$ be the automorphic induction of $\tau_\xi$ which is a cuspidal automorphic representation of $GL_2(\mathbb{A}_F)$ so that $L(s,\pi)=L(s,\rho_\xi)$. Note that the irreducibility of $\rho_\xi$ 
implies the cuspidality of $\pi$. 
It follows from Lemma \ref{lem:292} that $\rho_\xi$ is totally odd and then we see that 
the Weil-Deligne representation $WD_\infty(\rho_\xi)$ associated to $\rho_\xi$ at $\infty$ is described as  
$$WD_\infty(\rho_\xi):W_{\mathbb{R}}=\mathbb{C}^\times\rtimes \{1,j\}\longrightarrow GL_2(\mathbb{C}),\ \mathbb{C}\ni z\mapsto I_2,\ j\mapsto 
\left(
\begin{array}{cc}
1 & 0 \\
0 & -1
\end{array}\right)
$$      
(see the argument in Proposition 2.1 of \cite{KY} and also \cite{Cas}). 
It follows from this that for each infinite place $\infty$ of $F$, one has   
\begin{align*}
 \pi_\infty\simeq \pi(1,{\rm sgn}) &:= {\rm Ind}^{GL_2(\mathbb{R})}_{B(\mathbb{R})}(1\otimes {\rm sgn}) 
\end{align*}
where $B({\mathbb{R}})$ is the group of upper-triangular matrices in $GL_2({\mathbb{R}})$, $1 \otimes {\rm sgn}: B({\mathbb{R}}) \rightarrow {\mathbb{R}}^\times$ is the group homomorphism defined by $\left(
\begin{smallmatrix}
 a & * \\
 0 & d
\end{smallmatrix}
\right) \mapsto {\rm sgn}(d)$ and ${\rm Ind}^{GL_2(\mathbb{R})}_{B(\mathbb{R})}(1\otimes {\rm sgn})$ is the parabolic induction by $1 \otimes {\rm sgn}$. This implies $\pi$ generated by a holomorphic Hilbert modular form of parallel weight one.  
\end{proof}
By Lemmas \ref{lem:292} and \ref{lem:wt-one}, we know that there exists a classical weight one cuspidal eigenform $f_\xi$ such that $\rho_\xi$ is equivalent to the representation $\rho_{f_\xi}$ associated to $f_\xi$. We shall observe that the $p$-stabilization(s) of $f_\xi$ can be obtained from the Hida community containing ${\mathcal F}$. More precisely, we are going to establish the following 
\begin{proposition}
\label{prop:211}
Let $f_\xi$ be as in Lemma $\ref{lem:wt-one}$. Then there exists a primitive $p$-ordinary cuspidal Hida family ${\mathcal F}_\xi$ which specializes to the $p$-stabilization(s) of $f_\xi$. Moreover, if $\xi$ is unramified at the prime ideals of $K$ lying over ${\mathfrak n}_0$, then the tame level of ${\mathcal F}_\xi$ is ${\mathfrak n}_0$ and hence ${\mathcal F}_\xi$ belongs to the same Hida community as ${\mathcal F}$. 
\end{proposition}
\begin{proof}
For each $i=1, \ldots, t$, since $\varphi$ and $\xi$ are unramified at ${\mathcal P}_i^\sigma$, $\varphi \xi$ is ramified at ${\mathcal P}_i$ if and only if $N_{K/F}({\rm cond}(\varphi \xi))$ is divisible by ${\mathfrak p}_i$. Conductor-discriminant formula ${\rm cond}(\rho_{f_\xi})=N_{K/F}({\rm cond}(\varphi \xi))d_{K/F}$ implies that these equivalent conditions hold true if and only if ${\rm cond}(\rho_{f_\xi})$ is divisible by ${\mathfrak p}_i$. As the conductor of $\rho_{f_\xi}$ is equal to the level of $f_\xi$, $\varphi \xi$ is ramified at ${\mathcal P}_i$ if and only if the level of $f_\xi$ is divisible by ${\mathfrak p}_i$. If this is the case, we have $(\varphi \xi)({\mathcal P}_i)=0$ and $(\varphi \xi)({\mathcal P}_i^\sigma) \neq 0$, and thus the normalized Fourier coefficient $c({\mathfrak p}_i, f_\xi)=(\varphi \xi)({\mathcal P}_i^\sigma)$ is a root of unity. Therefore $f_\xi$ is stabilized and ordinary at ${\mathfrak p}_i$. Suppose, on the contrary, that $\varphi \xi$ is unramified at ${\mathcal P}_i$. Then we have
\begin{align*}
 {\rm det}(XI_2-\rho_{f_\xi}({\rm Frob}_{{\mathfrak p}_i})) &= (X-(\varphi \xi)({\mathcal P}_i))(X-(\varphi \xi)({\mathcal P}_i^\sigma))
\end{align*}
and both $(\varphi \xi)({\mathcal P}_i)$ and $(\varphi \xi)({\mathcal P}_i^\sigma)$ are roots of unity. Therefore the two ${\mathfrak p}_i$-stabilizations of $f_\xi$ are both ordinary at ${\mathfrak p}_i$. In any case we can conclude that there is a $p$-stabilized newform $f_\xi^*$ such that the associated representation is equivalent to $\rho_\xi$. Then by a theorem of Wiles (Theorem 3; p.~532 of \cite{W88}) there exists a primitive $p$-ordinary cuspidal Hida family ${\mathcal F}_\xi$ specializing to $f_\xi^*$. In particular the residual representation $\bar{\rho}_{{\mathcal F}_\xi}$ is equivalent to $\bar{\rho}_{\mathcal F}$ by construction. 

The tame level of ${\mathcal F}_\xi$ is equal to the prime-to-$p$ part of the level of $f_\xi$. If $\xi$ is unramified at the prime ideals of $K$ sitting above ${\mathfrak n}_0$, then the prime-to-$p$ part of $N_{K/F}({\rm cond}(\varphi \xi))$ is equal to that of $N_{K/F}({\rm cond}(\varphi))$. We apply conductor-discriminant formula to see that the prime-to-$p$ part of ${\rm cond}(\rho_{f_\xi})$ and that of ${\rm cond}(\rho_f)$ coincide. Thus the prime-to-$p$ part of the level of $f_\xi$ and that of $f$ are the same. Since ${\mathcal F}$ is primitive of tame level ${\mathfrak n}_0$, the prime-to-$p$ part of the level of $f$ is ${\mathfrak n}_0$. Therefore the prime-to-$p$ part of the level of $f_\xi$ is also ${\mathfrak n}_0$, which shows the assertion. 
\end{proof}

\subsection{The proof of the main theorem}
\label{sec:204}
Now we complete the proof of Theorem \ref{theo:205}. 
\begin{proof}
Suppose that the ray class group ${\rm Cl}_K({\mathfrak n}_0{\mathcal Q}^\infty)$ is an infinite group. Then Lemma \ref{lem:208} tells us that there are infinitely many $p$-power order characters $\xi: G_K \rightarrow \overbar{\mathbb{Q}}_p^\times$ that are unramified outside ${\mathcal Q}$ and the infinite places and $\xi \circ {\rm ver}_{K/F}: G_F \rightarrow \overbar{\mathbb{Q}}_p^\times$ is unramified outside $p$. Each of such characters $\xi$ produces a $p$-stabilized newform of weight one which is obtained from the Hida community $\left\{\mathcal F \right\}$ and the associated Galois representation $\rho_\xi$ satisfies $\rho_\xi \cong \varepsilon_{K/F} \otimes \rho_\xi$. Hence the community $\left\{\mathcal F \right\}$ has infinitely many classical weight one specializations that have multiplication by $K$. Since a Hida community is a finite set, there exists a member ${\mathcal G}$ of $\left\{\mathcal F \right\}$ which admits infinitely many such specializations. With Theorem \ref{theo:107} in mind, we conclude that $K/F$ is totally imaginary and the family ${\mathcal G}$ has CM by $K$. The main theorem is proved. 
\end{proof}

\subsection{Families residually of CM type}
\label{sec:205}
As we declared in Remark \ref{rmk:291} (\ref{rmk:293}), we end this section by giving an upper bound when ${\mathcal F}$ is residually of CM type. Let ${\mathcal F}$ be a non-CM primitive $p$-ordinary Hida family such that $\bar{\rho}_{\mathcal F} \cong {\rm Ind}_K^F(\bar{\varphi})$ for some totally imaginary quadratic extension $K/F$ and a character $\bar{\varphi}: G_K \rightarrow {\mathbb{F}}^\times$. Since ${\mathcal F}$ is not of CM type, there is at least one prime ideal ${\mathfrak l}$ in $F$ that is inert in $K$, prime to ${\mathfrak n}_0p$ and the trace ${\rm Tr}\rho_{\mathcal F}({\rm Frob}_{\mathfrak l})$ is non-zero (see the proof of Theorem \ref{theo:107}). Therefore ${\rm Tr}\rho_{\mathcal F}({\rm Frob}_{\mathfrak l})$ is contained in only finitely many height one prime ideals of $\Lambda_L$. In particular, there are only finitely many height one prime ideals of $\Lambda_L$ lying over $P_{1, \varepsilon}$ for some $p$-power order character $\varepsilon: {\mathbf{G}} \rightarrow \overbar{\mathbb{Q}}_p^\times$. Let $\lambda_{{\mathcal F}, {\mathfrak l}}$ denote this number and put $\lambda_{\mathcal F}={\rm min}\left\{\lambda_{{\mathcal F}, {\mathfrak l}} \mid \text{${\mathfrak l}$ is inert in $K$ and prime to ${\mathfrak n}_0p$} \right\}$. 
\begin{proposition}[cf. Lemma 6.5; p.~682 of \cite{DG}]
\label{prop:294}
Let ${\mathcal F}$ be a non-CM family that is residually of CM type by $K$ in the above sense. Then the number of classical weight one specializations of ${\mathcal F}$ is bounded by $\lambda_{\mathcal F}$. 
\end{proposition}

\section{Families residually of exceptional type}
We give an upper bound on the number of classical weight one forms in a primitive $p$-ordinary non-CM cuspidal Hida family ${\mathcal F}$ residually of exceptional type. Let $p^r$ be the $p$-part of the class number $|{\rm Cl}_F|$ of $F$ and ${\mathfrak p}_1, \ldots, {\mathfrak p}_t$ the prime ideals of $F$ lying over $p$. 
\begin{theorem} 
\label{theo:301}
If ${\mathcal F}$ is residually of exceptional type, then ${\mathcal F}$ has at most $a \cdot b$ classical weight one specializations, where
\begin{itemize}
\item $a=1$, except $p=5$ and the type of ${\mathcal F}$ is $A_5$, in which case $a=2$, and
\item $b=p^r$, except $p=3$ or $5$, in which case $b=2^t \cdot p^r$. 
\end{itemize}
In particular, ${\mathcal F}$ has at most one weight one specialization if $p \geq 7$ and the class number $|{\rm Cl}_F|$ is not divisible by $p$. 
\end{theorem}
\begin{proof}
The proof relies on that of Theorem 5.1; p.~676 of \cite{DG}. We note that the Galois representation $\rho_f: G_F \rightarrow GL_2(O)$ attached to a classical weight one specialization $f$ of ${\mathcal F}$ (if any) is irreducible and hence determined by the trace. Therefore it is enough to show that there are at most $a$ choices for the projective trace ${({\rm Tr}\rho_f)^2}/{{\rm det}\rho_f}$, and for each of them there are at most $b$ choices for the determinant of $\rho_f$. Indeed, since ${\rm Tr}\rho_f$ is congruent to ${\rm Tr}\bar{\rho}_{\mathcal F}$ and $p$ is odd, ${\rm Tr}\rho_f$ is uniquely determined by $({\rm Tr}\rho_f)^2$. 

Since ${\mathcal F}$ is residually of $A_4$, $S_4$ or $A_5$ type, the image of $\bar{\rho}_{\mathcal F}(g)$ in $PGL_2({\mathbb{F}})$ has order at most five. A standard computation shows that the projective trace ${{\rm Tr}\bar{\rho}_{\mathcal F}(g)^2}/{{\rm det}\bar{\rho}_{\mathcal F}(g)}$ of $\bar{\rho}_{\mathcal F}(g)$ varies as in \textsc{Table \ref{table1}}, according to the order of $\bar{\rho}_{\mathcal F}(g)$ in $PGL_2({\mathbb{F}})$: 
\begin{table}[h]
\caption{} \label{table1}
  \begin{tabular}{|c||c|c|c|c|c|} \hline
    order of $\bar{\rho}_{\mathcal F}(g)$ in $PGL_2({\mathbb{F}})$ & $1$ & $2$ & $3$ & $4$ & $5$ \\ \hline
    projective trace of $\bar{\rho}_{\mathcal F}(g)$ & $4$ & $0$ & $1$ & $2$ & a root of $X^2-3X+1$ \\ \hline
  \end{tabular}
\end{table}
\newline Therefore, if ${\mathcal F}$ is residually of $A_4$ or $S_4$ type, there is at most one choice for the projective trace ${({\rm Tr}\rho_f)^2}/{{\rm det}\rho_f}$ of $f$. As for $A_5$ case, the two roots of $X^2-3X+1$ are congruent modulo $p$ if and only if $p=5$. Thus there are at most two choices for ${({\rm Tr}\rho_f)^2}/{{\rm det}\rho_f}$. 

It remains to be shown that there are at most $b$ choices for the determinant of ${\rm det}(\rho_f)$ with a given projective trace. Since ${\mathcal F}$ is $p$-ordinary, we have
\begin{align*}
 \rho_f|_{I_{\mathfrak p}} & \cong \left(
\begin{array}{cc}
 {\rm det}(\rho_f) & * \\
 0 & 1
\end{array}
\right)
\end{align*}
for each prime ${\mathfrak p}$ of $F$ lying over $p$. This implies that $\rho_f(I_{\mathfrak p})$ injects into $PGL_2(O)$. Therefore the order of ${\rm det}(\rho_f)=\psi_{\mathcal F} \varepsilon$ restricted to $I_{\mathfrak p}$ is at most five, where $\varepsilon: {\mathbf{G}} \rightarrow \overbar{\mathbb{Q}}_p^\times$ is the $p$-power order character such that $\varphi_{1, \varepsilon}({\mathcal F})=f$. Recall that $\psi_{\mathcal F}$ is tamely ramified at ${\mathfrak p}$ and thus $\psi_{\mathcal F}|_{I_{\mathfrak p}}$ has order prime to $p$. If $p \geq 7$ this implies that $\varepsilon|_{I_{\mathfrak p}}=1$ for all ${\mathfrak p}$ lying over $p$. Consequently $\varepsilon$ is unramified everywhere and hence is a character of the ideal class group of $F$. This proves the assertion when $p \geq 7$. Throughout the rest of the proof, we assume $p=3$ or $5$. Suppose that $f=\varphi_{1, \varepsilon}({\mathcal F})$ and $f'=\varphi_{1, \varepsilon'}({\mathcal F})$ are two classical weight one forms in ${\mathcal F}$ having the same projective trace. Then there exists a $p$-power order character $\eta: G_F \rightarrow \overbar{\mathbb{Q}}_p^\times$ so that $\rho_f \cong \eta \otimes \rho_{f'}$. This immediately yields the equality $\varepsilon=\eta^2 \varepsilon'$. Note that $\eta$ is uniquely determined by $f'$. Since $\varepsilon$ and $\varepsilon'$ are unramified outside $p$, so is $\eta^2$. Moreover $\eta$ is of $p$-power order and $p$ is odd. This reveals that $\eta$ is unramified outside $p$. We investigate the behavior of $\eta$ at each prime ${\mathfrak p}$ lying over $p$. The relation $\rho_f \cong \eta \otimes \rho_{f'}$ restricted to $I_{\mathfrak p}$ implies that at least one of the following occurs: 
\begin{enumerate}
\item $\eta=1$ and $\varepsilon'=\varepsilon$ on $I_{\mathfrak p}$; 
\item $\eta=\varepsilon=(\varepsilon')^{-1}$ and $\psi_{\mathcal F}=1$ on $I_{\mathfrak p}$. 
\end{enumerate}
Let $S$ be the set of primes ideals ${\mathfrak p}$ lying over $p$ such that $\varepsilon'|_{I_{\mathfrak p}} \neq \varepsilon|_{I_{\mathfrak p}}$. Then $\eta$ is unramified outside $S$ and $\eta|_{I_{\mathfrak p}}=\varepsilon|_{I_{\mathfrak p}}$ for each ${\mathfrak p} \in S$. Let $\zeta: G_F \rightarrow \overbar{\mathbb{Q}}_p^\times$ be a $p$-power order character and suppose that $\eta \zeta$ satisfies the same ramification condition as that of $\eta$. Then $\zeta$ is unramified everywhere so there are $p^r$ choices for $\zeta$ (recall that $p^r$ is the $p$-part of the class number $|{\rm Cl}_F|$). Obviously the set $S$ depends on $\varepsilon'$ and there are at most $2^t$ choices for $S$. Hence ${\mathcal F}$ has at most $2^t \cdot p^r$ weight one specializations, which proves the theorem. 
\end{proof}
The following proposition is a generalization of Proposition 5.2; p.~678 of \cite{DG}: 
\begin{proposition}
\label{prop:303}
Let $p \geq 7$ be a prime number that splits completely in $F$ and $\left\{{\mathcal F} \right\}$ is a Hida community of exceptional type which is $p$-distinguished and the residual representation $\bar{\rho}_{\mathcal F}$ is absolutely irreducible when restricted to ${\rm Gal}(\overbar{F}/F(\zeta_p))$. Assume further that the tame level ${\mathfrak n}_0$ of $\left\{{\mathcal F} \right\}$ is the same as the conductor of $\bar{\rho}_{\mathcal F}$. Then $\left\{{\mathcal F} \right\}$ has at least one classical weight one specialization $f$. Moreover, any other classical weight one specialization of $\left\{{\mathcal F} \right\}$ can be written as $f \otimes \eta$, where $\eta: G_F \rightarrow {\mathbb{Q}}_p^\times$ is a $p$-power order character of conductor dividing ${\mathfrak n}_0$. 
\end{proposition}
\begin{proof}
Since $p \geq 7$ the order of the image of $\bar{\rho}_{\mathcal F}$ in $GL_2({\mathbb{F}})$ is prime to $p$. Therefore we can take the Teichm\"uller lift $\tilde{\rho}: G_F \rightarrow GL_2(W({\mathbb{F}}))$ of $\bar{\rho}_{\mathcal F}$, where $W({\mathbb{F}})$ is the ring of Witt vectors of ${\mathbb{F}}$. The reduction map $GL_2(W({\mathbb{F}})) \rightarrow GL_2({\mathbb{F}})$ induces an isomorphism on $\tilde{\rho}(G_F)$. This implies that $\tilde{\rho}$ is an Artin representation and $\tilde{\rho}(G_F)$ is a semi-simple group. As ${\mathcal F}$ is $p$-ordinary, $\bar{\rho}_{\mathcal F}|_{D_{\mathfrak p}}$ is a direct sum of the characters $\bar{\mathcal E}_{\mathfrak p}$ and $\bar{\mathcal D}_{\mathfrak p}$ for each prime ${\mathfrak p}$ lying over $p$. Therefore $\tilde{\rho}|_{D_{\mathfrak p}}$ is also a direct sum of two characters. We apply the modularity lifting theorem below to obtain a classical weight one form $f$ such that $\rho_f \cong \tilde{\rho}$: 
\begin{theorem}[Theorem 3; p.~999 of \cite{Sa10}]
\label{theo:209}
Let $p$ be an odd prime that splits completely in $F$. Let $E$ be a finite extension of ${\mathbb{Q}}_p$ with the ring of integers $O$ and the maximal ideal ${\mathfrak m}$, and $\rho: G_F \rightarrow GL_2(O)$ a continuous totally odd representation satisfying the following conditions: 
\begin{itemize}
\item $\rho$ ramifies at only finitely many primes; 
\item $\bar{\rho}=(\rho \bmod {\mathfrak m})$ is absolutely irreducible when restricted to ${\rm Gal}(\bar{F}/F(\zeta_p))$, and has a modular lifting which is potentially ordinary and potentially Barsotti-Tate at every prime of $F$ above $p$; 
\item For every prime ${\mathfrak p}$ of $F$ above $p$, the restriction $\rho|_{D_{\mathfrak p}}$ is the direct sum of $1$-dimensional characters $\chi_{{\mathfrak p}, 1}$ and $\chi_{{\mathfrak p}, 2}$ of $D_{\mathfrak p}$ such that the images of the inertia subgroup at ${\mathfrak p}$ are finite and $(\chi_{{\mathfrak p}, 1} \bmod {\mathfrak m}) \not \equiv (\chi_{{\mathfrak p}, 2} \bmod {\mathfrak m})$. 
\end{itemize}
Then there exists an embedding $\iota: E \rightarrow \overbar{{\mathbb{Q}}}_p \cong {\mathbb{C}}$ and a classical cuspidal Hilbert modular eigenform $f$ of weight one such that $\iota \circ \rho: G_F \rightarrow GL_2({\mathbb{C}})$ is isomorphic to the representation associated to $f$ by Rogawski-Tunnell \cite{RT}. 
\end{theorem}
By a theorem of Wiles (Theorem 3; p.~532 of \cite{W88}) there exists a Hida family ${\mathcal G}$ specializing to (the $p$-stabilization(s) of) $f$. The tame level of ${\mathcal G}$ is equal to the prime-to-$p$ part of the level of $f$. By our assumption ${\rm cond}(\rho_f)={\rm cond}(\tilde{\rho})={\rm cond}(\bar{\rho}_{\mathcal F})$ is equal to ${\mathfrak n}_0$. 

As for the second claim, let $f$ and $g$ be two classical weight one specializations of the community $\left\{{\mathcal F} \right\}$, say $f=\varphi_{1, \varepsilon}({\mathcal F})$ and $g=\varphi_{1, \varepsilon'}({\mathcal G})$ for members ${\mathcal F}, {\mathcal G} \in \left\{{\mathcal F} \right\}$ and $p$-power order characters $\varepsilon, \varepsilon': {\mathbf{G}} \rightarrow \overbar{\mathbb{Q}}_p ^\times$. As the projective image of $\rho_f$ and $\rho_g$ are equivalent, there exists a $p$-power order character $\eta: G_F \rightarrow \overbar{\mathbb{Q}}_p^\times$ so that $\rho_g \cong \eta \otimes \rho_f$. By determinant considerations, one can show, in a fashion similar to the proof of Theorem \ref{theo:301}, that $\eta$ is unramified outside ${\mathfrak n}_0p$. Let ${\mathfrak p}$ be a prime of $F$ lying over $p$. The relation $\rho_g \cong \eta \otimes \rho_f$ restricted to $I_{\mathfrak p}$ implies that at least one of the following holds true: 
\begin{enumerate}
\item $\eta=1, \psi_{\mathcal F}=\psi_{\mathcal G}$ and $\varepsilon'=\varepsilon$ on $I_{\mathfrak p}$; 
\item $\eta=\varepsilon^{-1}=\varepsilon'$ and $\psi_{\mathcal F}=\psi_{\mathcal G}=1$ on $I_{\mathfrak p}$. 
\end{enumerate}
Recall that $\varepsilon|_{I_{\mathfrak p}}=\varepsilon'|_{I_{\mathfrak p}}=1$ provided $p \geq 7$ (see the proof of Theorem \ref{theo:301}). Thus in both cases we have $\eta|_{I_{\mathfrak p}}=1$. Therefore $\eta$ is unramified outside ${\mathfrak n}_0$, and the conductor of $\eta$ divides ${\mathfrak n}_0$ since $\psi_{\mathcal G}\varepsilon'=\eta^2 \psi_{\mathcal F}\varepsilon$. 
\end{proof}
We end this article by making two remarks on Proposition \ref{prop:303}. 
\begin{remark}
\label{rmk:304}
\begin{enumerate}
\item In the proposition, we assumed that $p$ splits completely in $F$, $\left\{{\mathcal F} \right\}$ is $p$-distinguished, and the residual representation $\bar{\rho}_{\mathcal F}$ is absolutely irreducible when restricted to ${\rm Gal}(\overbar{F}/F(\zeta_p))$. This is due to our application of modularity lifting theorem in the course of the proof. 
\item Under the assumptions of Proposition \ref{prop:303}, the community $\left\{{\mathcal F} \right\}$ has a unique classical weight one specialization if the order of the narrow ray class group ${\rm Cl}_F({\mathfrak n}_0)$ of $F$ of modulus ${\mathfrak n}_0$ is prime to $p$. 
\end{enumerate}
\end{remark}

\section*{Acknowledgements}
The author would like to express her gratitude to her supervisor, Nobuo Tsuzuki, for his helpful advice and unceasing encouragement. She also thanks Takuya Yamauchi for useful information, suggestions and discussions. The author is partially supported by Japan Society for the Promotion of Science (JSPS), Grant-in-Aid for Scientific Research for JSPS fellows (15J00944). 


\end{document}